\documentclass[11pt]{article}

\usepackage{amsmath,amsthm,verbatim,amssymb,amsfonts,amscd, graphicx}
\usepackage{hyperref}
\usepackage{graphics}
\usepackage{esint}
\usepackage{bm}
\usepackage{color}

\theoremstyle{plain}
\newtheorem{theorem}{Theorem}

\newtheorem{lemma}{Lemma}

\theoremstyle{definition}

\newcommand{\R}{\mathbb{R}}

\newcommand{\B}{\mathbb{B}}
\newcommand{\Bn}{\mathbb{B}^n}

\newcommand{\Sn}{\mathbb{S}^{n-1}}

\newcommand{\diam}{\text{diam}}

\newcommand{\barint}{\mathop{\hbox{\vrule height3pt depth-2.7pt width.65em}\hskip-1em \int}}

\newcommand{\red}{\color{red}}

\newcommand{\Mod}{\textnormal{Mod}}

\numberwithin{definition}{section}
\numberwithin{equation}{section}
\numberwithin{theorem}{section}
\numberwithin{lemma}{section}

\begin{document}
 
 \author{Sita Benedict, Pekka Koskela and Xining Li}
  \title{Weighted Hardy Spaces of Quasiconformal Mappings}
  \date{}
\maketitle

\begin{abstract}
We establish a weighted version of the  $H^p$-theory of
quasiconformal mappings. 
\end{abstract}

\section{Introduction}
 Let $f:\Bn\rightarrow\R^n$ be a quasiconformal mapping, see Section 2.2 for the definition. Analogously to the setting of analytic functions defined in the unit disk, we say that $f$ belongs to Hardy space $H^p,0<p<\infty$, provided that
$$(*)\sup_{0<r<1}\int_{\Sn}|f(r\omega)|^pd\sigma(\omega)<\infty.$$
The theory of quasiconformal Hardy spaces was initiated in \cite{Z}. 
According to Beurling's theorem, for a given quasiconformal mapping $f$, 
the radial limit
 $$f(\omega)=\lim_{r\rightarrow 1}f(r\omega)$$
exists for almost every $\omega \in \Sn.$ Zinsmeister used this result in 
\cite{Z} to charcterize membership in $H^p$ via $L^p$-integrability of 
radial limits and via $L^p$-integrability of a nontangential maximal function. 
For further results on quasiconformal $H^p$-spaces, we refer the reader to 
\cite{AK}.
 
 The theory of quasiconformal $H^p$-spaces is a generalization of the theory 
of $H^p$-spaces of univalent functions. In the latter setting, one may employ 
the powerful machinery of analytic functions. Especially, in \cite{PR} this 
machinery was utilized towards a weighted theory of $H^p$-spaces of univalent 
functions. Let us define $M(r,f)=\sup_{\omega\in \Sn}|f(r\omega)|$ for $0<r<1$. 
Then $f$ belongs to $H^p$ if and only if $$\int_0^1 M(r,f)^pdr<\infty.$$ 
In  \cite{PR},  the weighted Hardy space  for $-1<\alpha<\infty$ and $0<p<\infty$ was defined as the class of all univalent functions for 
which
 $$(**)\int_0^1 M(r,f)^p(1-r)^\alpha<\infty.$$
 Notice that $$\sup_{0<r<1}(1-r)^\alpha \int_{\Sn}|f(r\omega)|^pd\sigma(\omega)=\infty$$
 for any univalent function when $\alpha<0$ and that 
$\lim_{r\rightarrow 1}(1-r)^{\alpha}|f(r\omega)|=\infty$ for almost every $\omega$ when $\alpha<0$. Consequently, one cannot give a simple definition for weighted Hardy spaces based on a variant of $(*)$ or on weighted radial limits. Several equaivalent characterizations for membership in weighted Hardy spaces were given 
in \cite{GPaPe} and \cite{PR}.

Given $0<p<\infty,-1<\alpha<\infty$ and a quasiconformal mapping $f:\B\rightarrow\R^n$ we write $f\in H^p_{\alpha}$ whenever
\begin{eqnarray}\label{p-maximal-modulus-estimate}
\int_0^1 (1-r)^{n-2+\alpha} M^p(r,f)dr < \infty.
\end{eqnarray}
We establish the following characterization of membership in $H^p_{\alpha}$.
\begin{theorem}\label{for all p}
Let $f: \B^n \rightarrow \R^n$ be a quasiconformal mapping, $0 < p < \infty$, and $-1 < \alpha < \infty$. Then the following are equivalent:
\begin{eqnarray}\label{p-weighted-hp}
||f||^p_{ H^p_{\alpha}}:=\int_0^1 (1-r)^{n-2+\alpha} M^p(r,f)dr < \infty.
\end{eqnarray}
\begin{eqnarray}\label{p-image-area-estimate}
\int_0^1 (1-r)^{n-2+\alpha}\int_{B(0,r)}|f(x)|^{p-n}|Df(x)|^n dxdr < \infty
\end{eqnarray}
\begin{eqnarray}\label{p-area-estimate}
\int_0^1 (1-r)^{n-2+\alpha}\left(\int_{B(0,r)}|Df(x)|^n dx\right)^{p/n}dr < \infty
\end{eqnarray}

 If $\alpha\ge 0$ or $p\ge n$, the above conditions are further equivalent to
\begin{eqnarray}\label{p-averaged-gradient-estimate}
\int_{\B^n}a_f^p(x)(1-|x|)^{p-1+\alpha} dx < \infty.
\end{eqnarray}
\end{theorem}

%

Actually, we prove a bit more than what is stated in Theorem \ref{for all p}. Namely, for $\alpha\ge 0$ \eqref{p-weighted-hp}-\eqref{p-averaged-gradient-estimate} are further equivalent to
$$\int_{\Sn}\left(\sup_{0<r<1}|f(r\omega)|^p(1-r)^{\alpha}\right)d\sigma$$
and to
$$\int_{\Sn}\left(\sup_{x\in \Gamma(\omega)}|f(x)|^p(1-|x|)^{\alpha}\right)d\sigma .$$
This together with Theorem \ref{for all p} gives a rather complete quasiconformal analog of the characterizations for the weighted $H^p$-spaces of univalent functions in \cite{PR}. 

Our proof of Theorem \ref{for all p} relies on techniques from \cite{AK} and \cite{Z}, but our weighted setting requires some new ideas. For example, the equivalence of \eqref{p-image-area-estimate} and \eqref{p-area-estimate} with the membership in quasiconformal $H^p$ is new even in the unweighted setting.

\vskip 10mm


\section{Background and Preliminaries}
\subsection{Cones and Shadows}\label{cones}

Given $x\in \Bn,$ we define
$$B_x=B(x,(1-|x|)/2)$$
and for $\omega\in\Sn$, we let
$$\Gamma(\omega)=\cup\{B_{r\omega}:0<r<1\}.$$
This is a cone with a tip at $\omega$. Finally, the shadow of $B_x$ is
$$S_x=\{\frac{z}{|z|}:0\neq z\in B_x\}.$$
It is easy to check that $x\in \Gamma(\omega)$ if $\omega\in S_x.$

\subsection{Quasiconformal Mappings}
Let $G\subset \R^n$ be a domain. We say that $f:G\rightarrow\R^n$ is 
a $K$-quasiconformal  mapping for $K\geq 1$ if  $f$ is continuous and 
one-to-one (hence a homeomorphism onto $f(G)$),
 $f\in W^{1,n}_{loc}(G, \R^n)$ and $|Df(x)|\leq K J_f(x)$ for almost 
every $x\in G$. For convenice, we write $f:\Bn\rightarrow\Omega$ below to 
specify that $f$ is defined on $\Bn$ with $f(\Bn)=\Omega.$

We continue with important properties of quasiconformal mappings. The 
following estimates  
can be deduced from \cite[Lemma 2.1]{AK}, also see \cite{V}.

\begin{lemma}\label{qcbasic1}
Let $f:\Bn\rightarrow\Omega$ be a $K$-quasiconformal mapping. There exists a constant $C=C(n,K)$ such that for all $x\in\Bn,$ we have
$$\diam(f(B_x))/C\le d(f(x),\partial\Omega)\le C\diam(f(B_x))\le C^2d(f(B_x),\partial\Omega),$$
{ and} $d(f(x),\partial\Omega))/C\le |f(y)-f(x)|\le Cd(f(x),\partial\Omega))$
for every $y\in\partial B_x.$
\end{lemma}
 
A quasiconformal mapping is only almost everywhere differentiable and hence
we will employ the the concept of averaged derivative 
$$a_f(x)=\exp[\int_{B_x}\log J_f(y)dy/(n|B_x|)].$$
If $f$ is a conformal mapping, then  $|Df|^n=J_f$ and especially $a_f=|Df|,$
see \cite{AG2} for details and \cite{AG1} for the origins of the definition.

The following lemmas is from \cite{AG2}.

\begin{lemma}\label{qcbasic2}
Let $f:\Bn\rightarrow\Omega$ be a $K$-quasiconformal mapping. There exists a constant $C(n,K)$, with
$$d(f(x),\partial\Omega)/C\le a_f(x)(1-|x|)\le Cd(f(x),\partial\Omega)$$
and $$\frac{1}{C}\barint_{B_x}|Df(y)|^ndy\le Ca_f(x)\le \barint_{B_x}|Df(y)|^ndy$$
for all $x\in\Bn.$
\end{lemma}

The following result is \cite[Lemma 2.5]{AK}.


\begin{lemma}\label{qcbasic3}
Let $f:\B^n\rightarrow\Omega$ be a $K$-quasiconformal mapping, and suppose that $u>0$ satisfies 
$$u(y)/C\le u(x)\le Cu(y)$$
for all $x\in\B^n$ and $y\in B_x$. Let $0<q\le n$ and $p\ge q$. Then
\[\int_{\B^n}a_f^pu\, dx\approx\int_{\B^n} a_f^{p-q}|Df(x)|^qu\, dx.\]
with constants only depending on $p,q,n,C,K$.
\end{lemma}

We continue with a useful estimate.

\begin{lemma}\label{afubini}
Let $f:\B^n\rightarrow\Omega$ be a $K$-quasiconformal mapping. Let $0<p<\infty$
and $\alpha\in \R.$ Then
$$ \int_{\Sn} \sup_{x\in\Gamma(\omega)}d(f(x),\partial \Omega)^p(1-|x|)^{\alpha}d\sigma\le C_1 \int_{\Sn} \sup_{x\in\Gamma(\omega)}a_f^p(x)(1-|x|)^{p+\alpha}d\sigma\le C_2 \int_{\B^n}a_f^p(x)(1-|x|)^{p+\alpha}dx$$
with constants $C_1,C_2$ that only depend on $n,K,p,\alpha.$
\end{lemma}
\begin{proof}  
By Lemma \ref{qcbasic2} and Lemma \ref{qcbasic1}
applied to  $x\in\Gamma(\omega)$ we have the 
estimate 
$$d(f(x),\partial \Omega)(1-|x|)^{\alpha/p} \le C_1a_f(x)(1-|x|)^{1+\alpha/p}
$$
$$\le C_2(\int_{B_x}a_f^p(y)dy)^{1/p}(1-|x|)^{1-n/p+\alpha/p}\le C_3(\int_{\Gamma(\omega)}a_f(x)^p(1-|x|)^{p-n+\alpha}dx)^{1/p}$$
with constants that only depend on $n,K,p,\alpha.$

For any integrable function $h$ on $\Bn$, by the Fubini theorem
$$\int_{\Bn}|h(x)|dx\approx \int_{\Sn}\int_{\Gamma(\omega)}|h(y)|(1-|y|)^{1-n}dyd\sigma.$$
Especially this holds for $h(x)=a_f^p(x)(1-|x|)^{p-1+\alpha}$ and the claim 
follows.
\end{proof}

A measure $\mu$ on $\Bn$ is called a Carleson measure if there is a constant 
$C_{\mu}$ such that 
$$
\mu(\Bn \cap B(\omega, r)) \leq C_{\mu}r^{n-1}
$$ 
for all $\omega \in \Sn$, $r> 0$. The following lemma, see 
\cite{Jo} and \cite[Lemma 5.6]{AK}, gives us a family of Carleson measures. 

\begin{lemma}\label{carlesonmeasure}
If $f$ is quasiconformal in $\Bn$, $0 < p < n$, and $f(x) \neq 0$ for all $x\in \Bn$, then the measure $\mu$ defined by $d\mu = |Df(x)|^p|f(x)|^{-p}(1-|x|)^{p-1}dx$ is a Carleson measure on $\Bn$.
\end{lemma}

\subsection{Modulus}
Given a collection of locally rectifiable curves $\Gamma$ in $\R^n$, the modulus of $\Mod(\Gamma)$ is defined as:
$$\Mod(\Gamma)=\inf\int_{\R^n}\rho^ndx,$$
where the infimum is taken for all nonnegative Borel functions $\rho$ such that $\int_{\gamma}\rho ds\ge 1$ when $\gamma\in \Gamma$.
Given a $K$-quasiconformal mapping $f:\Omega\rightarrow \R^n$, one has 
$\Mod(\Gamma)/C\le \Mod(f\Gamma)\le C\Mod(\Gamma)$,  where $C=C(K,n)$. See e.g.
\cite{V} for a proof. 

We recall two useful estimates, see \cite{V}. Given $E\subset \Sn,$ $0<r<1,$ 
and the 
family $\Gamma$ of radial segments joining $rE:=\{rx/|x|: x\in E\}$ and $E$, 
we have
$$\Mod(\Gamma)=\sigma(E)(\log(1/r))^{1-n},$$
where $\sigma (E)$ is the surface area of $E$. For an upper bound, we always 
have
$$\Mod(\Gamma)\le \frac{\omega_{n-1}}{\log(R/r)^{n-1}},$$ 
if each $\gamma\in\Gamma$ joins $S^{n-1}(x,r)$ to $S^{n-1}(x,R)$, $0<r<R$.

The following modulus estimate that can be found in \cite{AK} and \cite{Z}, 
is  one of our key tools. 

\begin{lemma}\label{shadowest}\cite[Lemma 4.2, Remark 4.3]{AK}
There exists a constant $C = C(n, K)$ such that if $f$ is K-quasiconformal in $\Bn$, $x\in \Bn$, $M > 1$, and $\alpha \geq 0$, then 
\begin{eqnarray*}
\sigma(\{\omega \in S_x : d(f(w), f(x))(1-|x|)^\alpha > Md(f(x), \partial f(\Bn))(1-|x|)^\alpha\}) \leq \frac{C \sigma(S(x))}{(\log M)^{n-1}}\,.
\end{eqnarray*}
\end{lemma}

\subsection{Nontangential and Radial Maximal Functions}
Given $p>0,\alpha\ge 0,$ we define the weighted radial maximal and 
nontangential maximal functions  by setting

\begin{eqnarray*}
&Mf_{p,\alpha}(\omega)=\sup_{0<r<1}|f(r\omega)|(1-r)^{\alpha/p},\omega\in\Sn\\
\text{and}\,&\,M^*f_{p,\alpha}(\omega)=\sup_{x\in\Gamma(\omega)}|f(x)|(1-|x|)^{\alpha/p},\omega\in\Sn
\end{eqnarray*}

Even though the nontangential maximal function can be larger than the radial 
one, we have the following estimate.
   
\begin{lemma}\label{maximal function estimate}
Let $f: \Bn\rightarrow \R^n$ be a $K$-quasiconformal mapping and let 
$0<p<\infty$ and $\alpha\ge 0$. There exists a constant 
$C=C(n,K,p,\alpha)$ such that 
\begin{eqnarray}\label{nontangential maximal estimate}
 \int_{\Sn}  (M^*_{p,\alpha}(\omega))^p d\sigma(\omega) \leq C(n,K) \int_{\Sn} 
(M_{p,\alpha}f(\omega))^p d\sigma(\omega).
\end{eqnarray}
\end{lemma}
\begin{proof}

Given $\omega\in\Sn$ and $x_0\in\Gamma(\omega)$, there exists $0<r_0<1$ 
such that $x_0\in B_{r_0\omega}$. By the definition of $B_{r_0\omega}$, we have 
$\frac{1}{2}(1-r_0)\le 1-|x_0|\le 2(1-r_0)$ and 
$|r_0\omega-x_0|\le \frac 1 2 (1-r_0).$ Set $r_1=(1+r_0)/2.$ Then $r_1\omega\in 
\partial B_{r_0\omega}.$
Hence Lemma \ref{qcbasic1} gives that
\begin{equation}\label{blaa}
|f(r_0\omega)-f(x_0)|\le C(n,K)|f(r_0\omega)- f(r_1\omega)|.
\end{equation}
By the triangle inequality
\begin{equation}\label{blaa1}
|f(x_0)|\le |f(r_0\omega)|+|f(r_0\omega)-\partial f(x_0)|
\end{equation}
and
\begin{equation}\label{blaa2}
|f(r_0\omega)- f(r_1\omega)|\le |f(r_0\omega)|+|f(r_1\omega)|.
\end{equation}
By combining \eqref{blaa},\eqref{blaa1},\eqref{blaa2} we
obtain
$$
|f(x_0)|\le C(n,K)\left(|f(r_0\omega)|+|f(r_1\omega)|\right).$$
Since $1-r_1=(1-r_0)/2,$ we conclude that 
$$(M^*_{p,\alpha}f(\omega))^p\le C(n,K,\alpha) (M_{p,\alpha}f(\omega))^p,$$
and \eqref{nontangential maximal estimate} follows.
%
\end{proof}

\vskip 10mm

\section{Proof of Theorem \ref{for all p}}
 We begin with the following lemma. We will employ it in the proof of Lemma \ref{weightedmaxmod} to fix the problem that $M(r,f)(1-r)^{\alpha/p}$ for $\alpha>0$
need not be nondecreasing even though $M(r,f)$ is.

\begin{lemma}\label{NMlemma}
Let $M: [0,1) \rightarrow [0, \infty)$ be increasing and continuous with $M(0) = 0$. Let $p > 0$, $\alpha \geq 0$ and define $N(r) = \sup_{0 \leq t \leq r}M(t)(1-t)^{\alpha/p}$. Then 
\begin{eqnarray}\label{Mpint}
\int_0^1(1-r)^{n-2+\alpha}M^p(r)dr < \infty
\end{eqnarray}
if and only if
\begin{eqnarray}\label{Npint}
\int_0^1(1-r)^{n-2}N^p(r)dr < \infty.
\end{eqnarray}
\end{lemma}
\begin{proof}
Since 
\begin{eqnarray*}M^p(r)(1-r)^{n-2+\alpha} &=& M^p(r)(1-r)^{\alpha}(1-r)^{n-2} \\ &\leq&\left(\sup_{0 \leq t \leq r}M^p(t)(1-t)^\alpha\right)(1-r)^{n-2} = N^p(r)(1-r)^{n-2},
\end{eqnarray*}
we have that (\ref{Npint}) { implies} (\ref{Mpint}) for any $p$. 

 For the other direction, we may assume that $N(r)$ is unbounded. Moreover,
if the desired conclusion is true for the case $p=1$ and all $M$ as in our
formulation, then by applying it to
$\widehat{M}(r): = M^p(r)$ we obtain our claim for all $p>0$. 
So it suffices prove that \eqref{Mpint} implies \eqref{Npint} for $p=1$. 

We define a sequence of points $r_k \in [0, 1)$ as follows. Let $r_0 = 0$ and set ${r_k = \inf\{r : N(r) = 2^{k-1} \}}.$ Then the continuity and monotonicity 
of $N(r)$ gives that $2^{k-1}=N(r_k) = M(r_k)(1-r_k)^\alpha$. Hence
\begin{eqnarray*}
\int_0^1 N(r) (1-r)^{n-2}dr &\leq& \sum_{k=0}^\infty N(r_{k+1})\int_{r_k}^{r_{k+1}} (1-r)^{n-2}dr\\ &=& \sum_{k=0}^\infty \frac{N(r_{k+1})}{n-1}\left[ (1-r_k)^{n-1} - (1-r_{k+1})^{n-1}\right] \\ &=& \sum_{k=0}^\infty \frac{M(r_{k+1})(1-r_{k+1})^\alpha}{n-1}\left[ (1-r_k)^{n-1} - (1-r_{k+1})^{n-1}\right]
\\ &=& \frac{2}{n-1}\sum_{k=0}^\infty M(r_k)(1-r_k)^\alpha\left[ (1-r_k)^{n-1} - (1-r_{k+1})^{n-1}\right] \\ &=& \frac{2}{n-1}\sum_{k=0}^\infty M(r_k)\left[ (1-r_k)^{n-1+\alpha} - (1-r_k)^\alpha(1-r_{k+1})^{n-1}\right] \\ &\leq& \frac{2}{n-1}\sum_{k=0}^\infty M(r_k)\left[ (1-r_k)^{n-1+\alpha} - (1-r_{k+1})^{n-1+\alpha}\right]
\end{eqnarray*}
We also have that 
\begin{eqnarray*}
\int_0^1(1-r)^{n-2+\alpha}M(r)dr &\geq& \sum_{k=0}^\infty M(r_k) \int_{r_k}^{r_{k+1}}(1-r)^{n-2+\alpha}dr \\&=&  \sum_{k=0}^\infty \frac{M(r_k)}{n-1+\alpha} \left[(1-r_k)^{n-1+\alpha} - (1-r_{k+1})^{n-1+\alpha} \right]. 
\end{eqnarray*}
\noindent The desired implication follows.
\end{proof}

We continue with a result on Carleson measures.

\begin{lemma}\label{carlesonlemma}
Let $f: \Bn \rightarrow \R^n$ be a quasiconformal mapping, 
 $0 < p < \infty$, $\alpha \geq 0$ and { let} $\mu$ { be} a Carleson measure on $\Bn$. Then there is a constant $C = C(n,K, C_\mu)$ such that
\begin{eqnarray*}
\int_{\Bn} |f(x)|^p(1-|x|)^\alpha d\mu \leq C \int_{\Sn} (Mf_{p,\alpha})^p d\sigma(\omega).
\end{eqnarray*}
\end{lemma}

\begin{proof}
By Lemma~\ref{maximal function estimate} it suffices to show that there 
exists a constant $C(n, K)$ such that
\begin{eqnarray*}
\int_{\Bn} |f(x)|^p(1-|x|)^\alpha d\mu \leq C \int_{\Sn} (M^*f_{p,\alpha})^p d\sigma(\omega).
\end{eqnarray*}
For each $\lambda > 0$, set
$$
E_\lambda = \{x\in \Bn : |f(x)|(1-|x|)^{\alpha/p} > \lambda \} 
$$
and
$$
U_\lambda = \left\{\omega\in \Sn : \sup_{x\in \Gamma(\omega)}|f(x)|(1-|x|)^{\alpha/p} > \lambda \right\}. 
$$
Recall the definion of the shadows $S_x$ from Subsection \ref{cones}. They
are spherical caps.
We can decompose $U_\lambda$ into a Whitney-type decomposition of 
spherical caps. That is, we can write, 
$$
U(\lambda) = \bigcup_{k=1}^\infty S_{x_k},
$$
where any $\omega \in U_\lambda$ belongs to no more than 
$N(n)$ spherical caps $S_{x_k}$ and 
$$d(S_{x_k}, \partial U(\lambda)) \approx \text{diam}(S_{x_k}) 
\approx (1- |x_k|),$$
with universal constants. If $x \in E_\lambda$ and $x\neq 0$, then 
$M^*f_{p,\alpha}(\omega)>\lambda$ whenever $x\in \Gamma_{\omega}.$ Moreover, 
$\frac{x}{|x|}\in S_{x_k}$ for some $k.$ Thus by the definition of 
$S_k$ and the properties of the Whitney-type decomposition 
there exists a universal constant $C$ such that $1-|x|\le C(1-|x_k|).$
Hence $E_\lambda \subset \cup_{k=1}^\infty B(x_k/|x_k|, C(1-|x_k|))$. Therefore
\begin{eqnarray*}
\mu(E_\lambda) &\leq& \sum_{k=1}^\infty \mu(B(x_k/|x_k|, C(1-|x_k|)) \cap \Bn) \\ &\leq& C(n, C_\mu) \sum_{k=1}^\infty (1-|x_k|)^{n-1} \\ &\leq& C(n, C_\mu)\sum_{k=1}^\infty \sigma (S_{x_k}) \leq C(n, C_\mu)\sigma(U_\lambda).
\end{eqnarray*}
This together with the Cavalieri formula gives
\begin{eqnarray*}
\int_{\Bn} |f(x)|^p(1-|x|)^\alpha d\mu &=& \int_0^\infty p\lambda^{p-1}\mu(E_\lambda)d\lambda \\&\leq& C(n, C_\mu) \int_0^\infty p\lambda^{p-1}\sigma(U_\lambda)d\lambda \\ &=& C(n, C_\mu) \int_{\Sn} \sup_{x\in \Gamma(\omega)}|f(x)|^p(1-|x|)^\alpha d\sigma(\omega)\\
 &=&  C \int_{\Sn} (M^*f_{p,\alpha}(\omega))^p d\sigma(\omega).
\end{eqnarray*}
%
\end{proof}

We are now ready to prove a maximal characterization for $H^p_{\alpha}.$
By Lemma~\ref{maximal function estimate} we could also replace the radial maximal function by the nontangential one.

\begin{lemma}\label{weightedmaxmod}
Let $f: \B^n \rightarrow \R^n$ be a quasiconformal mapping, 
$0 < p < \infty$ and $\alpha \geq 0$. Then
\begin{eqnarray}\label{supint}
\int_{\Sn} \sup_{0 < r < 1}|f(r\omega)|^p(1-r)^\alpha d\sigma(\omega)< \infty
\end{eqnarray}
if and only if
\begin{eqnarray}\label{Mineq}
\int_0^1 (1-r)^{n-2+\alpha}M^p(r,f) dr < \infty.
\end{eqnarray}
\end{lemma}

\begin{proof}
Assume $f(0) = 0$ and suppose (\ref{Mineq}) holds.  
Set $N(r,f) = \sup_{0 \leq t \leq r}M(t,f)(1-t)^{\frac{\alpha}{p}}$. 
Then by Lemma \ref{NMlemma} we have that\begin{eqnarray*}
\int_0^1 N^p(r,f)(1-r)^{n-2}dr < \infty.
\end{eqnarray*}
Recall our notation 
$Mf_{p,\alpha}(\omega) = \sup_{0 < r < 1}|f(r\omega)|(1-r)^{\frac{\alpha}{p}}$. Now
\begin{eqnarray}\label{integraali}
\int_{\Sn} \sup_{0 < r < 1}|f(r\omega)|^p(1-r)^\alpha d\sigma(\omega) &=& \int_{\Sn} Mf^p_{p,\alpha}(\omega) d\sigma(\omega)\nonumber\\ &=& \int_0^\infty p \lambda^{p-1} \sigma(\{\omega \in \Sn : Mf_{p, \alpha}(\omega) > \lambda\})d\lambda.
\end{eqnarray}
Fix $\lambda > 0$ and let 
$E = \{\omega \in \Sn : Mf_{p, \alpha}(\omega) > \lambda\}$. 
Suppose that $E$ is nonempty. Then there is $\omega \in \Sn$ and $r\in (0,1)$ such that $N(r,f) = \frac{\lambda}{2}$, since $N(r,f)$ is continuous. Our function $N(r,f)$ is also nondecreasing and we let 
\begin{equation}\label{rlamda}
r_\lambda = \max\{r : N(r,f) = \lambda/2\}.
\end{equation}
 We may assume that $\lambda$ is so large that $1/2 <r_\lambda<1.$ 
Let $\Gamma_E$ be the family of radial line segments connecting $B(0,r_\lambda)$ and $E \subset \Sn$. Then 
$$
\textnormal{Mod}(\Gamma_E) = \sigma(E)(\ln(1/r_\lambda))^{1-n} \geq \sigma(E)2^{1-n}(1-r_\lambda)^{1-n}.
$$
By the definitions of $E$ and $r_\lambda$, for any $\gamma \in \Gamma_E$, the image curve $f(\gamma)$ connects $B(0, (\lambda/2)(1-r_\lambda)^{-\alpha/p})$ to $\R^n\setminus B(0, \lambda(1-r_\lambda)^{-\alpha/p})$, and therefore the modulus of the image family $f\Gamma_E$ satisfies 
$$\textnormal{Mod}(f\Gamma_E) \leq \sigma(\Sn)(\ln 2)^{1-n}.$$ 
{ By }combining the { above} two estimates and using the quasi-invariance of the modulus, we arrive at the upper bound
$$
\sigma(E) \leq C(n,K)(1-r_\lambda)^{n-1}.
$$

In order to prove \eqref{supint} we may assume that $Mf_{p, \alpha}$ is unbounded
on $\Sn.$ Define a measure $\nu$ on $[0,1]$ by  setting  
$d\nu = (1-r)^{n-2}dr$ and recall the definition of $r_{\lambda}$ from
\eqref{rlamda}.
Now 
\begin{eqnarray*}
\nu(\{r : N(r, f) > \lambda/2 \}) =\int_{r_\lambda}^{1}(1-r)^{n-2}dr
 =\frac{(1-r_\lambda)^{n-1}}{n-1}.
\end{eqnarray*}
 Thus
\begin{gather*}
 \int_0^\infty p \lambda^{p-1} \sigma(\{\omega \in \Sn : Mf_{p, \alpha}(\omega) > \lambda\})d\lambda \\ \leq \sigma(\Sn)2^pN^p(1/2,f)+  \int_{2N(1/2,f)}^\infty p \lambda^{p-1} \sigma(\{\omega \in \Sn : Mf_{p, \alpha}(\omega) > \lambda\})d\lambda \\ \leq \sigma(\Sn)2^pN^p(1/2,f)+ C(n,K,p)\int_0^\infty \lambda^{p-1}(1-r_\lambda)^{n-1}d\lambda \\ \le \sigma(\Sn)2^pN^p(1/2,f)+ C(n,K,p) \int_0^\infty \lambda^{p-1}\int_{\{r:N(r, f) > \lambda/2 \}} (1-r)^{n-2}drd\lambda \\ \le \sigma(\Sn)2^pN^p(1/2,f) + C(n,K,p)\int_0^1(1-r)^{n-2}N^p(r,f) dr < \infty,
\end{gather*}
and hence \eqref{supint} follows by \eqref{integraali}. 

In the case 
$f(0)\neq 0,$ we consider the quasiconformal mapping $g$ defined
by setting $g(x)=f(x)-f(0).$ Then \eqref{Mineq} also holds with $f$ replaced 
by $g,$ and by the first part of our proof \eqref{supint} follows with $f$ 
replaced by $g.$ We conclude with \eqref{supint} via the triangle inequality.

For the other direction, suppose that (\ref{supint}) holds, set 
$r_k: = 1-2^{-k}$ and choose $x_k \in \Bn$ so that $|x_k| = r_k$ and 
$|f(x_k)| = M(r_k, f)$. Then
\begin{gather*}
\int_0^1 (1-r)^{n-2+\alpha}M^p(r,f)dr \leq 2^n\sum_{k=1}^\infty (2^{-k})^{n-1+\alpha}M^p(r_k,f) = 2^n\int_{\Bn}|f(x)|^p(1-|x|)^\alpha d\mu,
\end{gather*}
where $d\mu = \sum_{k=1}^\infty (1-|x|)^{n-1}\delta_{x_k}.$ Notice that $\mu$ is a
Carleson measure.
Hence Lemma \ref{carlesonlemma} gives us that
\begin{eqnarray*}
\int_0^1 (1-r)^{n-2+\alpha}M^p(r,f)dr \leq C(n,K, C_\mu) \int_{\Sn} \sup_{0 < r < 1}|f(r\omega)|^p(1-r)^\alpha d\sigma(\omega).
\end{eqnarray*}
\end{proof}

We continue with the following estimate whose proof is based on a 
good-$\lambda$ inequality.

{\begin{lemma}
\label{goodlambdaie}
Let $f:\Bn \rightarrow \R^n$ be a $K$-quasiconformal mapping, $0<p<\infty$, and $\alpha\ge0$. Let $$v(\omega)=\sup_{x\in\Gamma(\omega)}d(f(x),\partial f(\Bn))(1-|x|)^{\alpha/p}\in L^p(\Sn).$$
There exists $C=C(n,K,p,\alpha)$ such that
$$\int_{\Sn}{ M}f_{p,\alpha}^p(\omega)d\sigma(\omega)
\le C\int_{\Sn}v^p(\omega)d\sigma(\omega).$$
\end{lemma}
\begin{proof}
Recall that 
\begin{eqnarray*}
M^*f_{p,\alpha}(\omega)=\sup_{x\in\Gamma(\omega)}|f(x)|(1-|x|)^{\alpha/p}\\
\text{and }\,\,Mf_{p,\alpha}(\omega)=\sup_{0<r<1}|f(r\omega)|(1-r)^{\alpha/p}.
\end{eqnarray*}
Let $L>2.$ By the Cavalieri formula
\begin{equation}\label{intergal of f(w)}
\int_{\Sn}Mf_{p,\alpha}^p(\omega)d\sigma(\omega)=
L^p\int_{0}^{\infty}p\lambda^{p-1}\sigma(\{\omega\in\Sn:Mf_{p,\alpha}(\omega)>L\lambda\})d\lambda.
\end{equation}
Set $\Sigma_\lambda=\sigma(\{\omega\in\Sn:{ M}f_{p,\alpha}(\omega)>L\lambda\}).$ 
Then, for any $\gamma>0$, we have
$$\Sigma_\lambda\le \sigma(\{\omega\in\Sn:{ M}f_{p,\alpha}(\omega)>L\lambda,v(\omega)\le \gamma\})+\sigma(\{\omega\in\Sn:v(\omega)>\gamma\}).$$
If $\gamma$ is a fixed multiple of $\lambda,$ then the latter term is
what we want but we need to obtain a suitable estimate for the first term.

Towards this end, set $$E_{L\lambda,\gamma}=\{\omega\in\Sn:{ M}f_{p,\alpha}(\omega)>L\lambda,v(\omega)\le \gamma\}$$ and
define
$$U(\lambda)=\{\omega\in\Sn:M^*f_{p,\alpha}(\omega)>\lambda\}.$$ Since
$L>2\ge 1,$ clearly $E_{L\lambda,\gamma}\subset U(\lambda).$
We utilize a generalized Whitney decomposition of the open set $U(\lambda)$
as in the proof of Lemma \ref{carlesonlemma}:
$$U(\lambda)=\cup S_{x_k}.$$
where the caps $ S_{x_k}$ have uniformly bounded overlaps and 
\begin{equation}\label{capsdistance}
d( S_{x_k},\partial U(\lambda))\approx \diam( S_{x_k})\approx(1-|x_k|).
\end{equation}
Suppose $\omega\in  S_{x_k}$ is such that $v(\omega)\le \gamma$ and 
${ M}f_{p,\alpha}(\omega)>L\lambda$. According to \eqref{capsdistance}, 
we can choose $\bar{\omega}{   \in\partial} U(\lambda)$ with 
\begin{equation}\label{baari}
d(\omega,\bar{\omega})\le C\diam(S(x_k)).
\end{equation} 
Let $\bar{x}_k\in \Gamma(\bar{\omega})$ satisfy $|\bar{x}_k|=|x_k|.$
By \eqref{capsdistance}, we conclude that $d(x_k,\bar{x}_k)\le C(1-|x_k|)$. 
Hence  Lemma \ref{qcbasic1} allows us to conclude that
\begin{equation}\label{ekayla}
d(f(x_k),f(\bar{x}_k))(1-|x_k|)^{\alpha/p}\le Cd(f(x_k),\partial\Omega)(1-|x_k|)^{\alpha/p}\le Cv(\omega)\le  C\gamma.
\end{equation}
Since $\bar \omega\notin U(\lambda),$ we may deduce from \eqref{ekayla} that
\begin{equation}\label{lambda and gamma}
|f(x_k)|(1-|x_k|)^{\alpha/p}\le (|f(\bar{x}_k)|+d(f(x_k,f(\bar{x}_k)))(1-|x_k|)^{\alpha/p}\le \lambda+C\gamma.
\end{equation}
Next, the assumption that ${ M}f_{p,\alpha}(\omega)>L\lambda,$ allows us to 
choose choose $r_\omega\in(0,1)$ such that 
\begin{equation}\label{fr0iso}
|f(r_\omega\omega)|(1-r_\omega)^{\alpha/p}\ge \frac{1}{2} Mf_{p,\alpha}(\omega)\ge \frac{1}{2}L\lambda.\end{equation}

We proceed to show that 
\begin{equation}\label{r0pieni}
1-r_\omega\le C_0 (1-|x_k|)
\end{equation} 
for an absolute constant 
$C_0.$ Suppose not. Then $1-|x_k|\le \frac{1}{C_0}(1-r_\omega)$, which implies by 
\eqref{baari} that 
$$d(w,\bar w)\le C\diam(S(x_k))\le C(1-|x_k|)\le \frac {C}{C_0}(1-r_\omega).$$
This shows that $r_\omega\omega\in \Gamma_{\bar \omega}$
when $C_0>2C.$  
Since $\frac L 2>1,$ we conclude that  
$$M^*f_{p,\alpha}(\bar{\omega})> \lambda,$$
which contradicts the assumption that $\bar \omega\notin U(\lambda).$


We may assume that $C_0\ge 1.$
By \eqref{fr0iso} together with \eqref{r0pieni} we obtain 
\begin{equation}\label{fiso1}
L\lambda\le 2|f(r_\omega\omega)|(1-r_0)^{\alpha/p}\le2 C_0^{\alpha/p} 
|f(r_\omega\omega)|(1-|x_k|)^{\alpha/p}. 
\end{equation}
Let us fix the value of $L$ by choosing $L= 4C_0^{\alpha/p}.$ 
Then \eqref{fiso1} yields
\begin{equation}\label{fiso2}
2\lambda\le |f(r_\omega\omega)|(1-|x_j|)^{\alpha/p}.
\end{equation}

We proceed to estimate $\sigma(S_{x_k}\cap E_{L\lambda,\gamma}).$ Let
$\omega \in S_{x_k}\cap E_{L\lambda,\gamma}.$ Then there is $r_{\omega}\in (0,1)$ so
that both \eqref{r0pieni} and \eqref{fiso2} hold. Consider the
collection of all the corresponding caps $S_{r_\omega\omega}.$ By the Besicovitch
covering
theorem we find a 
countable subcollection of these caps, say $S_{r_1\omega_1},S_{r_2\omega_2},...,$
so that 
\begin{equation}\label{peitto}
S_{x_k}\cap E_{L\lambda,\gamma}\subset \bigcup_j S_{r_j\omega_j}
\end{equation}
and $\sum_j \chi_ {S_{r_j\omega_j}}(w)\le C_n$ for all $\omega\in \Sn.$
By \eqref{r0pieni}
we further have 
\begin{equation}\label{eipaha}
\sum_j\sigma (S_{r_j\omega_j})\le C_1 \sigma(S_{x_k})
\end{equation} for an absolute constant $C_1.$

Fix one of the caps  $S_{r_j\omega_j}=:S_j$ and let $A\ge 1.$ Write
$$E^j_1(A)=\{w\in S_j\cap S_{x_k}\cap E _{L\lambda,\gamma}:\ 
|f(w)-f(r_j\omega_j)|\ge Ad(f(r_j\omega_j),\partial \Omega)\}$$
and 
$$E^j_2(A)=\{w\in S_j\cap S_{x_k}\cap E _{L\lambda,\gamma}:\ 
|f(w)-f(x_k)|\ge Ad(f(x_k),\partial \Omega)\}.$$
We claim that we can find a constant $C_2$ only depending on $C_0,p,\alpha$ so 
that the choice $\lambda=C_2A\gamma$ guarantees that
\begin{equation}\label{yhdiste}
S_j\cap S_{x_k}\cap E_{L\lambda,\gamma}=E^j_1(A)\cup E^j_2(A).
\end{equation}

Let $\omega\in S_j\cap S_{x_k}\cap E_{L\lambda,\gamma}.$ 
Suppose first that
\begin{equation}\label{eka}
|f(\omega)-f(r_j\omega_j)|(1-r_j)^{\alpha/p}\ge A\gamma.
\end{equation}
Since  $\omega\in E_{L\lambda,\gamma}\cap S_j,$ we have 
$$\gamma \ge d(f(r_j\omega_j),\partial\Omega)(1-r_j)^{\alpha/p},$$ 
and we deduce from \eqref{eka} that $\omega\in E^j_1(A).$
We are left to consider the case
\begin{equation}\label{toka}
|f(\omega)-f(r_j\omega_j)|(1-r_j)^{\alpha/p}< A\gamma.
\end{equation}
%
%
Under this condition, the triangle inequality together with \eqref{r0pieni},
\eqref{fiso1} and
\eqref{lambda and gamma} give
\begin{equation}\begin{split}\label{vihdoin}
&|f(\omega)-f(x_k)|(1-|x_k|)^{\alpha/p}\ge |f(\omega)|(1-|x_k|)^{\alpha/p}-|f(x_k)|(1-|x_k|)^{\alpha/p}\\
\ge&(|f(r_j\omega_j)|-|f(\omega)-f(r_j\omega_j)|)\frac{(1-r_j)^{\alpha/p}}{C_0^{\alpha/p}}-|f(x_k)|(1-x_k)^{\alpha/p}\\
\ge& \frac{L\lambda}{2C_0^{\alpha/p}}-\frac{A\gamma}{C_0^{\alpha/p}}-(\lambda+C\gamma)
\ge2\lambda-(\lambda+C\gamma)-\frac{A\gamma}{C_0^{\alpha/p}}.
\end{split}\end{equation}
We now fix the relation between $\lambda$ and $\gamma$ by setting $\lambda=(C+\frac{A}{C_0^{\alpha/p}}+1)\gamma.$ Then \eqref{vihdoin} reduces to
$$|f(\omega)-f(x_k)|(1-|x_k|)^{\alpha/p}\ge A\gamma\ge Ad(f(x_k),\partial\Omega)(1-|x_k|)^{\alpha/p}$$
and we conclude that $\omega\in E^j_2(A).$ 

According to Lemma \ref{shadowest}, 
\begin{equation}\label{e1varjo}
\sigma(E^j_1(A))\le  \frac {C_2\sigma(S_j)}{(\log A)^{n-1}},
\end{equation}
where $C_2$ depends only on $K,n.$
Thus \eqref{e1varjo} together with \eqref{eipaha} gives
\begin{equation}\label{e1yla}
\sum_j\sigma(E^j_1(A))\le \frac {C_1C_2\sigma(S_{x_k})}{(\log A)^{n-1}}.
\end{equation}
We also deduce via Lemma \ref{shadowest}
that
\begin{equation}\label{e2yla}
\sigma(\cup_j E^j_2(A)) \le \sigma(\{\omega \in S_{x_k}:\ |f(w), f(x_k)|\ge Ad(f(x), \partial \Omega)\}\le \frac {C_2\sigma(S_{x_k})}{(\log A)^{n-1}}.
\end{equation}
Now \eqref{yhdiste} together with \eqref{e1yla} and \eqref{e2yla} gives
\begin{equation}\label{oyla}
\sigma(S_{x_k}\cap E_{L\lambda,\gamma})\le \frac {C_3\sigma(S_{x_k})}{(\log A)^{n-1}},
\end{equation}
where $C_3$ depends only on $K,n.$

By the choice of the caps $S_{x_k},$ the definition of $E_{L\lambda,\gamma}$ and 
\eqref{oyla} give via summing over $k$ the estimate
\begin{equation}\label{kokoylaraja}
\begin{split}
\Sigma_\lambda\le&\sigma(E_{L\lambda,\gamma})+\sigma(\{\omega\in\Sn:v(\omega>\gamma)\})\\
\le& \frac{C_3\sigma(U(\lambda))}{(\log A)^{n-1}}+\sigma(\{\omega\in\Sn:v(\omega>\gamma)\}).
\end{split}
\end{equation}

We insert \eqref{kokoylaraja} into \eqref{intergal of f(w)} and conclude that
\begin{equation}\label{viela}
\begin{split}
\int_{\Sn}{M}f_{p,\alpha}^p(\omega)d\sigma(\omega)&= L^p\int_{0}^{\infty}p\lambda^{p-1}\Sigma_\lambda d\lambda\\
&\le L^p\int_{0}^{\infty}p\lambda^{p-1}\frac{C_3\sigma(U(\lambda))}{(\log A)^{n-1}} d\lambda+ L^p\int_{0}^{\infty}p\lambda^{p-1}\sigma(\{\omega\in\Sn:v(\omega>\gamma)\})d\lambda\\
&\le \frac{C_3L^p}{(\log A)^{n-1}}\int_{\Sn}{M^*_{p,\alpha}f^p}(\omega)d\omega+
L^p\int_{0}^{\infty}p\lambda^{p-1}\sigma(\{\omega\in\Sn:v(\omega>\gamma)\})d\lambda.
\end{split}
\end{equation}
Suppose that the integral on the left-hand-side of \eqref{viela} is finite.
Then Lemma \ref{maximal function estimate} allows us to choose $A$ only
depending on $K,n,p,\alpha,L,C_3$ so that the integral of $M^*_{p,\alpha}f^p$
can be embedded into the left-hand-side. In this case our claim follows
via the Cavalieri formula, recalling that 
$\lambda=(C+\frac{A}{C_0^{\alpha/p}}+1)\gamma.$
We are left with the case where the integral on the  
left-hand-side of \eqref{viela} is infinite. In this case, we replace $f$ by
the $K$-quasiconformal map $f^j$ defined by setting $f^j(x)=f((1-1/j)x).$ 
Since the corresponding integral is now finite, we obtain a uniform estimate
for the integral of $Mf^j_{p,\alpha}$ in terms of the integral of $v^j,$
defined analogously. The desired estimate follows via the Fatou lemma
by letting $j$ tend to
infinity since it easily follows that $v_j(\omega)\le v(\omega)$ for all $\omega$ and that $Mf^j_{p,\alpha}(\omega)\to Mf_{p,\alpha}(\omega)$ for a.e. $\omega.$ 
\end{proof}}

\begin{lemma}\label{energychar}
Let $f: \B^n \rightarrow \R^n$ be a quasiconformal mapping, $0 < p < \infty$ and $\alpha \geq 0$. Then the following are equivalent:
\begin{enumerate}
\item$ \int_{\Sn} { Mf^p_{p,\alpha}}d\sigma < \infty$
\item$\int_{\B^n} a_f^p(x)(1-|x|)^{p-1+\alpha}dx < \infty$
\item $\int_{\Sn} \sup_{x\in\Gamma(\omega)}a_f^p(x)(1-|x|)^{p+\alpha}d\sigma < \infty$
\end{enumerate}
\end{lemma}

\begin{proof}
$\bf(1\implies 2)$ 
Suppose first that $0<p\le 1.$ We may assume that $f\neq 0$ in $\Bn.$ Then the
measure given by $d\mu=|Df|^p|f|^{-p}(1-|x|)^{p-1}dx$ is a Carleson measure by
Lemma \ref{carlesonmeasure} and hence Lemma \ref{qcbasic3} and
Lemma \ref{carlesonlemma} give
\begin{align*}
&\int_{\B^n} a_f^p(x)(1-|x|)^{p-1+\alpha}dx\le C\int_{\B^n} |Df|^p(1-|x|)^{p-1+\alpha}dx\\
\le C& \int_{\B}|f(x)|^p(1-|x|)^{\alpha}d\mu(x)\le C \int_{\Sn} Mf^p_{p,\alpha}(\omega)d\sigma.
\end{align*}

Suppose finally that $p>1$ and pick $y\in \partial f(\Bn)$.
By Lemma \ref{qcbasic3} and Lemma \ref{qcbasic2}, we have
\begin{align*}
&\int_{\Bn}a_f^p(x)(1-|x|)^{p-1+\alpha}dx\le C\int_{\Bn}|Df|a_f^{p-1}(x)(1-|x|)^{p-1+\alpha}dx\\
\le C&\int_{\Bn}|Df|d(f(x),\partial f(\Bn))^{p-1}(1-|x|)^{\alpha}dx\le C\int_{\Bn}|Df||f(x)-y|^{p-1}(1-|x|)^{\alpha}dx.
\end{align*}
Since $f(x)-y\neq 0$ in $\Bn,$ the measure given by $d\mu=|Df(x)||f(x)-y|{^{\red -1}}dx$ is a Carleson measure Lemma \ref{carlesonmeasure}. Hence we can apply 
Lemma \ref{carlesonlemma} to conclude that
$\int_{\Bn}a_f^p(x)(1-|x|)^{p-1+\alpha}dx<\infty.$


$\bf(2\implies3)$ This follows from Lemma \ref{afubini}.

$\bf(3\implies 1)$ 
By Lemma \ref{qcbasic2} we have that $d(f(x),\partial \Omega)\le C
a_f(x)(1-|x|).$  Hence $\bf (1)$ follows from 
$\bf (3)$ by Lemma \ref{goodlambdaie}.
%
%
\end{proof}

\begin{lemma}\label{p>n,negalpha}
Let $f: \B^n \rightarrow \R^n$ be a quasiconformal mapping  with $f(0)=0.$ Let
$p\ge n$ and $-1 < \alpha$. Then 
$$
\int_0^1 (1-r)^{n-2+\alpha}M^p(r,f) dr \leq C \int_{\Bn} a_f(x)^p(1-|x|)^{p-1+\alpha}dx.
$$ 
\end{lemma}
 \begin{proof}


Define $v(\omega)=\sup_{x\in \Gamma(\omega)}d({   f}(x),\partial{   f}(\Bn))(1-|x|)^{\alpha/{   p}}.$ By Lemma \ref{afubini}
we only need to show that 
\begin{equation}\label{puuttuva1}
\int_{0}^{1}(1-r)^{n-2+\alpha}M(r,{   f})^{   p}dr\le C\int_{\Sn}v^{   p}d\sigma.
\end{equation}

For each $i\ge 1,$ let $r_i=1-2^{-i}$ and pick $x_i\in \Sn(r_i)$ with  $|f(x_i)|=M(r_i,{   f})$. 
{Then}
 \begin{equation}\label{puuttuva2}\begin{split}
\int_{0}^{1}(1-r)^{n-2+\alpha}{   M(r,f)^p}dr&=\sum_{i=1} ^{\infty}\int_{r_{i-1}}^{r_i}(1-r)^{n-2+\alpha}{   M(r,f)^p}dr\\
&\le C \sum_{i=1} ^{\infty}|{   f}(x_i)|^{   p} (1-|x_i|)^{n-1+\alpha}.
\end{split}\end{equation}
Let $\tilde{C}$ be a constant, to be determined later, and let
${ G}(f)=\{i\in \mathbb{N}:|{   f(x_i)}| \le \tilde{C} d({   f(x_i)},\partial f(\Bn))\}$ and $B(f)=\mathbb{N}\setminus { G}(f).$
For $i\in { G}(f)$ and $\omega\in S(x_i)$, we have 
\begin{equation}\label{lem-dirichilet-hardy2}
|{   f(x_i)}|^p (1-|x_i|)^{n-1+\alpha} \le \tilde{C}^{   p} d({   f(x_i)},\partial f(\Bn))^p(1-|x_i|)^{n-1+\alpha}\le \tilde{C}^{   p} v_f(\omega)^p(1-|x_i|)^{n-1}.
\end{equation}

Letting $\delta=1/\tilde{C}$, for $i\in B(f),$ we have $ d({   f(x_i)},\partial f(\Bn))\le \delta|{   f(x_i)}|$.
Set $\omega_i=\frac{x_i}{|x_i|}$, and let $y_{i-1}=r_{i-1}\omega_i$.  Then we have $x_i\in B_{y_{i-1}}$.
   Hence Lemma \ref{qcbasic1} gives
\begin{equation*}
\begin{split}
|f(x_i)-f(y_{i-1})|\le \diam f(B_{ y_{i-1}})\le{  C d(f(B_{y_{i-1}}),\partial\Omega)\le C^2 d(f(x_i),\partial\Omega)}.
\end{split}
\end{equation*}

Therefore, by the choice of $x_{i-1},$ we obtain 
$$|{   f}(x_i)|\le |{   f}(y_{i-1})|+ C^2\delta|f(x_i)|\le |f(x_{i-1})|+C^2\delta|f(x_i)|.$$
If $\tilde C$ is sufficiently large, then $C^{ 2}\delta< 1$ and we have
\begin{equation}\label{gi and gi-11}
|{   f}(x_i)|\le { \lambda}|{   f}(x_{i-1})|,
\end{equation}
where $\lambda=1/(1-C^2\delta).$ By multiplying both sides of 
\eqref{gi and gi-11} to
 $(1-|x_i|)^{(n-1+\alpha)/{   p}}$ and raising to power ${p}$, we  conclude that
\begin{equation}
\begin{split}
|{   f}(x_i)|^{   p}(1-|x_i|)^{n-1+\alpha}&\le { \lambda}^{   p}|{   f}(x_{i-1})|^{   p}(1-|x_i|)^{n-1+\alpha}\\
&={ \lambda}^{   p}|{   f}(x_{i-1})|^{   p}2^{-(n-1+\alpha)}(1-|x_{i-1}|)^{n-1+\alpha}.
\end{split}\end{equation}
Now, notice that $n-1+\alpha>0$ for $\alpha>-1$. By recalling that 
$\delta=1/\tilde{C}$  and $\lambda=1/(1-C^2\delta)$, we find 
$\tilde{C}{ =\tilde{C}(p,C)}$ big enough such that ${\lambda}^{   p}2^{-(n-1+\alpha)}<1$ and $\tilde{C}\ge C_0.$ Then there exists $c<1,$ 
such that
$$|{   f}(x_i)|^{   p}(1-|x_i|)^{n-1+\alpha}\le c|{   f}(x_{i-1})|^{   p}(1-|x_{i-1}|)^{n-1+\alpha}.$$
Since $x_0=0$ and $f(0)=0,$ we have $0\in G(f).$
If $i-1\in B({   f}),$ we repeat the above argument with $i$ replaced by $i-1$ 
and arrive at 
$$|{   f}(x_i)|^{   p}(1-|x_i|)^{n-1+\alpha}\le c^2|{   f}(x_{i-2})|^{   p}(1-|x_{i-2}|)^{n-1+\alpha}.$$
We repeat inductively until ${i-k}\in G(f):$ 
for each $i\in { B }({   f})$, there exists $k$ such that $l\in B({   f})$,  
for all $ i-k<l\le i,$ $i-k\in G({   f})$ ande 
$$|{   f}(x_i)|^{   p}(1-|x_i|)^{n-1+\alpha}\le c^k|{   f}(x_{i-k})|^{p}(1-|x_{i-k}|)^{n-1+\alpha}.$$
Therefore, 
\begin{equation}\label{lem-dirichilet-hardy3}
\begin{split}
\sum_{i=0}^{\infty}{   |f(x_i)|^p} (1-|x_i|)^{n-1+\alpha}& \le C \sum_{i\in  G(f)}{   |f(x_i)|^p} (1-|x_{i}|)^{n-1+\alpha}\\
&\le C\tilde{C}\sum_{i\in {G(f}}d({   f}(x_i),\partial {   f}(\Bn))^{   p} (1-|x_{i}|)^{\alpha}(1-|x_{i}|)^{n-1}.
\end{split}\end{equation} 
Set $S_i=S_{x_i}$ for $i\in { G}(f)$.
Since $|x_i|=1-2^{-i},$ the definition of the shadows $S_i$, gives that
$$\sum_{j>i}\sigma(S_j)\le \sigma(S_i).$$
Thus, we also have
$$\sigma(\{\omega\in S_i:\sum_{j>i}\chi_{S_j}(\omega)\ge 2\})\le \frac{1}{2}\sigma(S_i).$$
Hence there is $\hat{S}_i\subset S_i$ with $\sigma(\hat{S}_i)\ge \frac{1}{2}\sigma(S_i)=C_n(1-|x_i|)^{n-1}$ and so that no point of $\hat{S}_i$ belongs to more 
than one $S_j.$ Then $$\sum_{j\in G(f)}\chi_{\hat{S}_i}(\omega)\le 2.$$ By 
combining \eqref{lem-dirichilet-hardy2} and \eqref{lem-dirichilet-hardy3}, we 
finally obtain that
\begin{align*}
\sum_{i=0}^{\infty}|f(x_i)|^p(1-|x_i|)^{n-1+\alpha}\le \frac{C\tilde{C}}{C_n}\sum_{i\in S(f)}\int_{\hat{S}_i}v^p(\omega)d\sigma\le  \frac{C\tilde{C}}{C_n}\int_{\Sn}v^p(\omega)d\sigma.
\end{align*}
This together with \eqref{puuttuva2} gives \eqref{puuttuva1} and hence our
claim follows. 
\end{proof}

\vskip 10mm

\begin{proof}[\bf Proof of Theorem \ref{for all p}]
  We will first prove the equivalence of  \eqref{p-weighted-hp}, \eqref{p-image-area-estimate} and \eqref{p-area-estimate}.

We assume $f(0) = 0$ and handle the cases $0 < p \leq n$ and $p > n$ separately.

{\bf Case 1} First suppose $0 < p \leq n$. Then 
\begin{gather*}
\int_{B(0,r)} |f|^{p-n}|Df|^n dx \leq K\int_{B(0,r)}|f|^{p-n}J_f(x) dx = K\int_{f(B(0,r))}|y|^{p-n}dy \\ \stackrel{(*)}{\leq} K\int_{B(0, \sqrt[n]{|f(B(0,r))|})} |y|^{p-n}dy  = K\int_{\Sn}\int_0^{\sqrt[n]{(|f(B(0,r))}} t^{p-1} dt d\sigma = C(K,n,p)  |f(B(0,r))|^{p/n} \\ = C\left( \int_{B(0,r)} J_f(x) dx \right)^{p/n}  \leq C\left( \int_{B(0,r)} |Df(x)|^n dx \right)^{p/n},
\end{gather*}
where $(*)$ holds since the weight function $|y|^{p-n}$ is radially decreasing when $0<p\le n.$
We have proved that $\eqref{p-area-estimate} \Rightarrow \eqref{p-image-area-estimate}$. 

Now let $g = |f|^{(p-n)/n}f.$ Since quasiconformal mappings are differentiable almost everywhere, we can calculate that 
\begin{eqnarray*}
Dg = |f|^{(p-n)/n}\left(I + \frac{p-n}{n}\frac{f^Tf}{|f|^2}\right)Df\;\;\;\; (\textnormal {for a.e. $x \in \B^n$}),
\end{eqnarray*}
and so $|Dg| \lesssim  |f|^{(p-n)/n}|Df|$. Then Fubini's theorem and Lemma \ref{qcbasic3} give 
\begin{gather*}
\int_0^1 (1-r)^{n-2+\alpha}\int_{B(0,r)}|f|^{p-n}|Df|^n \gtrsim \int_0^1 (1-r)^{n-2 + \alpha} \int_{B(0,r)} |Dg|^n dx dr \\ 
{\approx} \int_{\B^n} |Dg|^n(1-|x|)^{n-1+\alpha}dx \approx \int_{\B^n} a_g^n(x)(1-|x|)^{n-1+\alpha}dx.
\end{gather*} 
So by assuming \eqref{p-image-area-estimate} in the statement of the theorem holds, we can apply Lemma \ref{p>n,negalpha} to the { quasiconformal mapping} $g$, which gives 

\begin{equation}\label{p-image-area-implies-maximal-modulus,p<n}
\begin{split}
&\int_0^1 (1-r)^{n-2+\alpha}M^p(r,f) dr \approx\int_0^1 (1-r)^{n-2+\alpha}M^n(r,g) dr\\
\lesssim & \int_{\Bn} a_g(x)^n(1-|x|)^{n-1+\alpha}dx\lesssim \int_0^1 (1-r)^{n-2+\alpha}\int_{B(0,r)}|f|^{p-n}|Df|^ndxdr.
\end{split}
\end{equation}
We have shown that $\eqref{p-image-area-estimate}\implies\eqref{p-weighted-hp}.$

The implication that $\eqref{p-weighted-hp} \implies \eqref{p-area-estimate}$ is a straightforward application of H\"older's inequality and the definition of $M(r,f)$. Indeed,
 \begin{equation}\begin{split}\label{p-maximal-modulus-implie-area}
&\int_0^1 (1-r)^{n-2+\alpha}\left(\int_{B(0,r)}|Df(x)|^n dx\right)^{p/n}dr  \leq \int_0^1 (1-r)^{n-2+\alpha}\left(\int_{B(0,r)}|Df(x)|^n|f(x)|^{p-n}M(r,f)^{n-p} dx\right)^{p/n}dr \\ = 
&\int_{0}^{1}(1-r)^{(n-2+\alpha)(n-p)/n}M(r,f)^{(n-p)p/n}\left((1-r)^{n-2+\alpha}\int_{B(0,r)}|Df(x)|^n|f(x)|^{p-n}dx\right)^{p/n}dr \\ \leq &\left(\int_{0}^{1}(1-r)^{n-2+\alpha}M(r,f)^pdr\right)^{(n-p)/n}\left(\int_{0}^{1}(1-r)^{n-2+\alpha}\int_{B(0,r)}|Df|^n|f|^{p-n}dxdr\right)^{p/n}.
\end{split}\end{equation}
Since
\begin{gather*}
\int_{B(0,r)}|Df|^n|f|^{p-n}dxdr \leq K\int_{f(B(0,r))}|y|^{p-n}dy \leq K \int_{B(0, M(r,f))} |y|^{p-n}dy \leq C(n,K,p) M(r,f)^p,
\end{gather*}
the desired result follows. 


{\bf Case 2} We assume {that} $p>n.$
{ Now}
\begin{align*}
\int_{B(0,r)}|Df|^ndx\le K\int_{B(0,r)}J_f(x)dx\le K|f(B(0,r))|\le K M(r,f)^n
\end{align*}
Therefore,  $\eqref{p-weighted-hp}$ { implies} $\eqref{p-area-estimate}$.
                       
Set $g=|f|^{(p-n)/n}f.$ Then $|g|^n=|f|^p,$ and $M(r,f)^p=M(r,g)^n.$ 
Analogously to \eqref{p-image-area-implies-maximal-modulus,p<n}, 
by Lemma \ref{p>n,negalpha}. We have
\begin{equation}\begin{split}\label{p-image-area-implies-maximal-modulus,p>n}
\int_0^1 (1-r)^{n-2+\alpha}M^p(r,f) dr\le \int_0^1 (1-r)^{n-2+\alpha}M^n(r,g) dr\ \lesssim \int_0^1 (1-r)^{n-2+\alpha}\int_{B(0,r)}|f|^{p-n}|Df|^ndxdr.
\end{split}
\end{equation}
Hence $\eqref{p-image-area-estimate}\implies \eqref{p-weighted-hp}.$

We need to show that $\eqref{p-area-estimate}\implies\eqref{p-image-area-estimate}.$ First of all,
\begin{equation}\label{p>n,1}
\begin{split}
&\int_{0}^{1}(1-r)^{n-2+\alpha}\int_{B(0,r)}|f|^{p-n}|Df|^ndxdr
\le \int_{0}^{1}(1-r)^{n-2+\alpha}M(r,f)^{p-n}\int_{B(0,r)}|Df|^ndxdr\\
\stackrel{\text{H\"older}}{\le}&(\int_{0}^{1}\left((1-r)^{n-2+\alpha}M(r,f))^{p}dr\right)^{(p-n)/p}\left(\int_{0}^{1}(1-r)^{n-2+\alpha}\left(\int_{B(0,r)}|Df|^ndx\right)^{p/n}dr\right)^{n/p}.
\end{split}
\end{equation}
If the later term on the right-hand-side is finite, then we obtain via
\eqref{p-image-area-implies-maximal-modulus,p>n} that
$$\int_{0}^{1}(1-r)^{n-2+\alpha}\int_{\Bn}|f|^{p-n}|Df|^ndxdr\le C\int_{0}^{1}(1-r)^{n-2+\alpha}\left(\int_{\Bn}|Df|^ndx \right)^{p/n}dr.$$
Since the constant in this inequality only depends on $n,K,p,\alpha$ the 
general case easily follows by applying this estimate with $f$ replaced by 
$f_j,$ defined by setting
$f_j(x)=f((1-1/j)x),$ and by passing to the limit.

We have shown the equivalence of \eqref{p-weighted-hp},\eqref{p-image-area-estimate} and \eqref{p-area-estimate} under the additional assumption that $f(0)=0.$
The general case follows since each of them holds for a $f$ if and only if
it holds for $g,$ defined by setting $g(x)=f(x)-f(0).$

{ We are left with the equivalence of \eqref{p-weighted-hp}--\eqref{p-averaged-gradient-estimate}, when $\alpha\ge 0$ or when $-1<\alpha<0$ with $p\ge n$.}

For $\alpha\ge 0,$ by Lemma \ref{weightedmaxmod} and Lemma \ref{energychar}, we { know} that { \eqref{p-weighted-hp} is equivalent to \eqref{p-averaged-gradient-estimate}.}

For $-1<\alpha<0,p\ge n$, by Lemma \ref{p>n,negalpha} { we know that \eqref{p-averaged-gradient-estimate} implies \eqref{p-maximal-modulus-estimate}, so}  we only need to show that $\eqref{p-image-area-estimate}$ { implies} $\eqref{p-averaged-gradient-estimate}$
Fix $y\in\partial f(\B)$.  Then, Lemma \ref{qcbasic2} and \ref{qcbasic3} give us that
\begin{equation}\begin{split}\label{p-maximal-implies-averaged-gradient,p>n}
\int_{\Bn}a_f(x)^p(1-|x|)^{p-1+\alpha}dx&\approx\int_{\Bn}a_f(x)^{p-n}(1-|x|)^{p-n}|Df(x)|^n(1-|x|)^{n-1+\alpha}dx\\
&\le C\int_{\Bn}d(f(x),\partial\Omega)^{p-n}|Df(x)|^n(1-|x|)^{n-1+\alpha}dx\\
&\le C\int_{\Bn}|f(x)-y|^{p-n}|Df(x)|^n(1-|x|)^{n-1+\alpha}dx\\
&\approx \int_{0}^{1}(1-r)^{n-2+\alpha}\int_{B(0,r)}|f(x)-y|^{p-n}|Df(x)|^ndxdr.
\end{split}\end{equation} 
 Notice that $|f(x)-y|\le |f(x)|+|y|$ and $|f(x)-y|^{p-n}\le C(p,n)(|f(x)|^{p-n}+|y|^{p-n})$. By \eqref{p-maximal-implies-averaged-gradient,p>n}, we need to show that
\begin{equation*}
\begin{split}
&\int_{0}^{1}(1-r)^{n-2+\alpha}\int_{B(0,r)}(|f(x)|^{p-n}+|y|^{p-n})|Df(x)|^ndxdr=\\
=&\int_{0}^{1}(1-r)^{n-2+\alpha}\int_{B(0,r)}|f(x)|^{p-n}|Df(x)|^ndxdr+|y|^{p-n}\int_{0}^{1}(1-r)^{n-2+\alpha}\int_{B(0,r)}|Df(x)|^ndxdr\\
=&(\uppercase\expandafter{\romannumeral1})+(\uppercase\expandafter{\romannumeral2})<\infty.
\end{split}
\end{equation*}
From the equivalenc of \eqref{p-maximal-modulus-estimate} and \eqref{p-image-area-estimate}, we know that $(\uppercase\expandafter{\romannumeral1})<\infty.$ On the other hand, we have
\begin{eqnarray*}
&\int_{0}^{1}(1-r)^{n-2+\alpha}\int_{B(0,r)}|Df(x)|^ndxdr\\
\le &\left(\int_{0}^{1}(1-r)^{n-2+\alpha}\left(\int_{B(0,r)}|Df(x)|^n\right)^{p/n}dxdr\right)^{n/p}\left(\int_{0}^{1}(1-r)^{n-2+\alpha}dr\right)^{(p-n)/p}.
\end{eqnarray*}
Apply the equivalence of \eqref{p-maximal-modulus-estimate} and \eqref{p-area-estimate} and $\int_{0}^{1}(1-r)^{n-2+\alpha}dr<\infty$ for $-1<\alpha<0.$ Then we have shown that $(\uppercase\expandafter{\romannumeral2})<\infty.$ Therefore, we have finished the proof of Theorem \ref{for all p}.
\end{proof}

\vskip .2cm

{\small
\noindent Addresses:\\

\noindent Sita Benedict: Department of Matematics, University of Hawaii, Honolulu, HI 96822, USA

\noindent E-mail:{\tt mawasigirl@gmail.com}

\vskip .2cm

\noindent Pekka Koskela: Department of Mathematics, University of Jyv\"askyl\"a, P.O. Box 35, FIN-40351 Jyv\"askyl\"a, Finland 

\noindent E-mail: {\tt pkoskela@math.jyu.fi}

\vskip .2cm

\noindent Xining Li: School of Mathematics, Sun Yat-Sen University, Guangzhou, China 510275

\noindent E-mail: {\tt lixining3@mail.sysu.edu.cn}
}

\end{document}